\UseRawInputEncoding
\documentclass[10pt]{article}
\usepackage[dvips]{color}
\usepackage{epsfig}
\usepackage{float,amsthm,amssymb,amsfonts}
\usepackage{ amssymb,amsmath,graphicx, amsfonts, latexsym}

\begin{document}
\theoremstyle{plain}
\newtheorem{theorem}{Theorem}
\newtheorem{corollary}[theorem]{Corollary}
\newtheorem{lemma}[theorem]{Lemma}
\newtheorem{proposition}[theorem]{Proposition}
\newtheorem{remark}[theorem]{Remark}

\theoremstyle{definition}
\newtheorem{defn}{Definition}
\newtheorem{definition}[theorem]{Definition}
\newtheorem{example}[theorem]{Example}
\newtheorem{conjecture}[theorem]{Conjecture}

\title{On the combinatorial and rank properties of certain subsemigroups of full
contractions  of a finite chain }
\author{\bf  M. M. Zubairu \footnote{Corresponding Author. ~~Email: $mmzubairu.mth@buk.edu.ng$} A. Umar and M. J. Aliyu \\[3mm]
\it\small  Department of Mathematical Sciences, Bayero  University Kano, P. M. B. 3011, Kano, Nigeria\\
\it\small  \texttt{mmzubairu.mth@buk.edu.ng}\\[3mm]
\it\small Khalifa University, P. O. Box 127788, Sas al Nakhl, Abu Dhabi, UAE\\
\it\small \texttt{abdullahi.umar@ku.ac.ae}\\[3mm]
\it\small Department of Mathematics, and Computer Sciences, Sule Lamido University, Kafin Hausa\\
\it\small \texttt{muhammadaliyu2@nda.edu.ng}
}
\maketitle\

\begin{abstract}
 Let $[n]=\{1,2,\ldots,n\}$ be a finite chain and let  $\mathcal{CT}_{n}$ be the  semigroup of full contractions on $[n]$. Denote $\mathcal{ORCT}_{n}$ and $\mathcal{OCT}_{n}$ to be the subsemigroup of order preserving or reversing and the subsemigroup of order preserving full contractions, respectively. It was shown in \cite{am} that the collection of all regular elements (denoted by, Reg$(\mathcal{ORCT}_{n})$ and Reg$(\mathcal{OCT}_{n}$), respectively) and the collection of all idempotent elements (denoted by E$(\mathcal{ORCT}_{n})$ and E$(\mathcal{OCT}_{n}$), respectively) of the subsemigroups $\mathcal{ORCT}_{n}$ and $\mathcal{OCT}_{n}$, respectively are subsemigroups. In this paper, we study some combinatorial and rank properties of these subsemigroups.
 \end{abstract}

\section{Introduction}
Denote $[n]=\{1,2,\ldots,n\}$ to be a finite chain and let  $\mathcal{T}_{n}$  denote the semigroup of full transformations of $[n]$. A transformation $\alpha\in \mathcal{T}_{n}$ is said to be \emph{order preserving} (resp., \emph{order reversing}) if  (for all $x,y \in [n]$) $x\leq y$ implies $x\alpha\leq y\alpha$ (resp., $x\alpha\geq y\alpha$); \emph{order decreasing} if (for all $x\in [n]$) $x\alpha\leq x$;  an \emph{isometry} (i.e., \emph{ distance preserving}) if (for all $x,y \in [n]$) $|x\alpha-y\alpha|=|x-y|$;   a \emph{contraction} if (for all $x,y \in [n]$) $|x\alpha-y\alpha|\leq |x-y|$. Let $\mathcal{CT}_{n}=\{\alpha\in \mathcal{T}_{n}: (\textnormal{for all }x,y\in [n])~\left|x\alpha-y\alpha\right|\leq\left|x-y\right|\}$ be the  semigroup of full contractions on $[n]$, as such $\mathcal{CT}_{n}$ is a subsemigroup of $\mathcal{T}_{n}$. Certain algebraic and combinatorial properties of this semigroup and some of its subsemigroups have been studied, for example see \cite{adu, leyla, garbac,kt, af, am, mzz, a1, a33}.

Let
 
 \noindent  \begin{equation}\label{ctn}\mathcal{OCT}_{n}=\{\alpha\in \mathcal{CT}_{n}: (\textnormal{for all}~x,y\in [n])~x\leq y \textnormal{ implies } x\alpha\leq y\alpha\},\end{equation} \noindent and
 \begin{equation}\label{orctn}\mathcal{ORCT}_{n}= \mathcal{OCT}_{n}\cup \{\alpha\in \mathcal{CT}_{n}: (\textnormal{for all}~x,y\in [n])~x\leq y ~ \textnormal{implies }  x\alpha\geq y\alpha\}\end{equation}
\noindent be the subsemigroups  of \emph{order preserving full contractions} and of \emph{order preserving or reversing  full contractions} on $[n]$, respectively. These subsemigroups are both known to be non-regular left abundant semigroups \cite{am} and their Green's relations have been characterized in \cite{mmz}. The ranks of $\mathcal{ORCT}_{n}$ and $\mathcal{OCT}_{n}$ were computed in \cite{kt} while the ranks of the two sided ideals of $\mathcal{ORCT}_{n}$ and $\mathcal{OCT}_{n}$ were computed by Leyla \cite{leyla}.

In 2021, Umar and Zubairu \cite{am} showed that the collection of all regular  elements (denoted by $\textnormal{Reg}(\mathcal{ORCT}_{n})$) of $\mathcal{ORCT}_{n}$ and also the collection of idempotent elements (denoted by $\textnormal{E}(\mathcal{ORCT}_{n})$) of $\mathcal{ORCT}_{n}$ are both subsemigroups of $\mathcal{ORCT}_{n}$. The two subsemigroups are both regular, in fact $\textnormal{Reg}(\mathcal{ORCT}_{n})$ has been shown to be an $\mathcal{L}-$ \emph{unipotent} semigroup (i.e., each ${L}-$class contains a unique idempotent). In fact, it was also shown in \cite{am} that the collection of all regular elements (denoted by Reg$\mathcal{OCT}_{n}$) in  $\mathcal{OCT}_{n}$ is a subsemigroup. However, combinatorial as well as rank properties of these semigroups are yet to be discussed, in this paper we discuss these properties, as such this paper is a natural sequel to Umar and Zubairu \cite{am}. For basic concepts in semigroup theory, we refer the reader to \cite{ maz, ph,howi}.

Let $S$ be a semigroup and $U$ be a subset of $S$, then $|U|$ is said to be the \emph{rank} of $S$ (denoted as $\textnormal{Rank}(S)$) if $$|U|=\min\{|A|: A\subseteq S \textnormal{ and } \langle A \rangle=S\}. $$

The notation $\langle U \rangle=S$ means that $U$ generate the semigroup $S$. The rank of several semigroups of transformation were investigated, see for example,  \cite{aj,ak2, gu, gu2, gu3, gm, mp}. However, there are several subsemigroups of full contractions which their ranks are yet to be known. In fact the order and the rank of the semigroup $\mathcal{CT}_{n}$ is still under investigation.

Let us briefly discuss the presentation of the paper . In section 1, we give a brief introduction and notations for proper understanding of the content of the remaining sections. In section 2, we discuss combinatorial properties for the semigroups $\textnormal{Reg}(\mathcal{ORCT}_n)$ and $\textnormal{E}(\mathcal{ORCT}_n)$, in particular we give their orders. In section 3, we proved that the rank of the semigroups $\textnormal{Reg}(\mathcal{ORCT}_n)$ and $\textnormal{E}(\mathcal{ORCT}_n)$ are 4 and 3, respectively, through the minimal generating set for their Rees quotient semigroups.

\section{Combinatorial Properties of $\textnormal{Reg}(\mathcal{ORCT}_n)$ and $\textnormal{E}(\mathcal{ORCT}_n)$ }
 In this section, we want to investigate some combinatorial properties of the semigroups,  $\textnormal{Reg}(\mathcal{ORCT}_n)$ and $\textnormal{E}(\mathcal{OCT}_n)$. In particular, we want to compute their Cardinalities. Let \begin{equation}\label{1}
    \alpha=\left( \begin{array}{cccc}
                           A_{1} & A_{2} & \ldots & A_{p} \\
                           x_{1} & x_{2} & \ldots & x_{p}
                         \end{array}
   \right)\in \mathcal{T}_{n}  ~~  (1\leq p\leq n),
    \end{equation}
then the \emph{rank} of $\alpha$ is defined and denoted by rank $(\alpha)=|\textnormal{Im }\alpha|=p$, so also, $x_{i}\alpha^{-1}=A_{i}$ ($1\leq i\leq p$) are equivalence classes under the relation $\textnormal{ker }\alpha=\{(x,y)\in [n]\times [n]: x\alpha=y\alpha\}$. Further, we denote the partition $(A_{1},\ldots, A_{p})$ by  $\textnormal{\textbf{Ker} }\alpha$ and also, fix$(\alpha)=|\{x\in[n]: x\alpha=x\}|$. A subset $T_{\alpha}$ of $[n]$ is said to be a \emph{transversal} of the partition $\textnormal{\textbf{Ker} }\alpha$ if $|T_{\alpha}|=p$, and $|A_{i}\cap T_{\alpha}|=1$ ($1\leq i\leq p$). A transversal  $T_{\alpha}$  is said to be \emph{convex} if for all $x,y\in T_{\alpha}$ with $x\leq y$ and if $x\leq z\leq y$ ($z\in [n]$), then $z\in T_{\alpha}$.

Before we proceed, lets describe some Green's relations on the semigroups $\textnormal{Reg}(\mathcal{ORCT}_n)$ and $\textnormal{E}(\mathcal{ORCT}_n)$. It is worth  noting that the two semigroups,  $\textnormal{Reg}(\mathcal{ORCT}_n)$ and $\textnormal{E}(\mathcal{ORCT}_n)$ are both regular subsemigroups of the Full Transformation semigroup $\mathcal{T}_n$, therefore by [\cite{howi}, Prop. 2.4.2] they automatically inherit the Green's $\mathcal{L}$ and $\mathcal{R}$ relations of the semigroup $\mathcal{T}_n$, but not necessary $\mathcal{D}$ relation, as such we have the following lemma.
\begin{lemma}
Let $\alpha,\beta \in S\in \{\textnormal{Reg}(\mathcal{ORCT}_n), \ \textnormal{E}(\mathcal{ORCT}_n)\}$, then
\begin{itemize}
\item[i] $\alpha \mathcal{R} \beta$ if and only if $\textnormal{Im }\alpha=\textnormal{Im }\beta$;
\item[ii] $\alpha \mathcal{L} \beta$ if and only if   $\textnormal{ker }\alpha=\textnormal{ker }\beta$.
\end{itemize}
\end{lemma}

\subsection{The Semigroup $\textnormal{Reg}(\mathcal{ORCT}_n)$}
Before we begin  discussing on the semigroup $\textnormal{Reg}(\mathcal{ORCT}_n)$, let us first of all consider the semigroup $\textnormal{Reg}(\mathcal{OCT}_n)$ consisting of only order-preserving elements.
 Let $\alpha$ be in $\textnormal{Reg}(\mathcal{OCT}_n)$, from [\cite{am}, Lem. 12], $\alpha$ is of the form $$\alpha=\left(\begin{array}{ccccc}
                                                                            \{1,\ldots,a+1\} & a+2 &  \ldots & a+p-1 & \{a+p,\ldots,n\} \\
                                                                            x+1 & x+2 & \ldots & x+p-1 & x+ p
                                                                          \end{array}
\right)$$\noindent 
Let

\begin{equation}\label{j} K_p=\{\alpha \in Reg(\mathcal{OCT}_n) : |\textnormal{Im }\alpha|=p\} \quad (1\leq p\leq n), \end{equation}

and suppose that $\alpha\in K_p$, by [\cite{az}, Lem. 12] Ker $ \alpha= \{\{1,\ldots,a+1\},a+2 \ldots, a+{p-1}, \{a+p,\ldots,n\} \}$  have an \emph{admissible} traversal (A transversal $T_{\alpha}$ is said to be {admissible} if and only if the map $A_{i}\mapsto t_{i}$  ($t_{i}\in T_{\alpha},\, i\in\{1,2,\ldots,p\}$) is a contraction, see \cite{mmz})  $T_\alpha= \{a+i\, : 1\leq i\leq p\}$ such that the mapping $a+i\mapsto x+i$ is an isometry. Therefore, translating the set $\{x+i :\, i\leq 1\leq p\}$  with an integer say $k$ to $\{x+i\pm k:\, 1\leq i\leq p\}$ will also serve as image set to $\textnormal{\textbf{Ker} }\alpha$ as long as $x+1-k\nless 1$ and $x+p +k \ngtr n$.

For example, if we define $\alpha$ as :
\begin{equation}\label{alf} \alpha= \left( \begin{array}{cccccc}
            \{1,\ldots,a+1\} & a+2& a_3 & \ldots & a+{p-1} & \{a+p,\ldots,n\} \\
            1 & 2 &  3& \ldots &p-1& p
          \end{array} \right).\end{equation}
then we will have $n-p$ other mappings in $K_p$ that will have the same domain as $\alpha$.

In similar manner, suppose we fix the image set $\{x+i |\, 1\leq i\leq p\}$ and consider $\textnormal{\textbf{Ker} }\alpha$, then we can refine the partition $\{\{1,\ldots,a+1\}, \{a+2\} \ldots, \{a+{p-1}\}, \{a+p,\ldots,n\} \}$ by $i-$shifting  to say $\{\{1,\ldots,a+i\}, \{a+i+1\} \ldots, \{a+{p-i}\}, \{a+p-i+1,\ldots,n\} \} $  for some integer $1\leq i\leq p $ which  also have an admissible convex traversal.

For the purpose of illustrations, if for some integer $j$, $\{\{1,\ldots,a+1\}, \{a+2\} \ldots, \{a+{p-1}\}, \{a+p,\ldots,n\} \}=\,\{\{1,2,\ldots j\}, \{j+1\}, \{j+2\}, \ldots, \{n\} \}$, then the translation $\{\{1,2,\ldots j-1\}, \{j\}, \{j+1\}, \ldots, \{n-1,n\} \}$ will also serve as domain to the image set of $\alpha$. Thus, for $p\neq 1$  we will have $n-p+1$ different mappings with the same domain set in $K_p$.

To see what we have been explaining, consider the table below; For $n\geq 4$, $2\leq p\leq n$ and $j=n-p+1$, the set $K_p$ can be presented as follows:

\begin{equation}\label{tabl}\resizebox{1\textwidth}{!}{$ \begin{array}{cccc}
     \left( \begin{array}{ccccc}
           \{1,\ldots j\}&j+1& \cdots &n-1& n \\
            1 & 2  & \ldots &p-1& p
          \end{array} \right)

          & \cdots  &

           \left( \begin{array}{ccccc}
           \{1,2\}&3& \cdots& \{p-1,\ldots n\} \\
            1 &   2& \ldots & p
          \end{array} \right)&

            \left( \begin{array}{ccccc}
           1&2& \cdots&p-1& \{p,\ldots n\} \\
            1 &   2& \cdots&p-1 & p
          \end{array} \right) \\

           \left( \begin{array}{ccccc}
           \{1,\ldots j\}&j+1& \cdots &n-1& n \\
            2 & 3 & \ldots &p& p+1
          \end{array} \right) & \cdots  &

           \left( \begin{array}{ccccc}
           \{1,2\}&3& \cdots& \{p-1,\ldots n\} \\
            2 &   3& \cdots & p+1
          \end{array} \right)&

           \left( \begin{array}{ccccc}
           1&2& \cdots&p-1& \{p,\ldots n\} \\
            2 &   3& \cdots&p & p+1
          \end{array} \right) \\ \vdots &\vdots& \vdots& \vdots

          \\
     \left( \begin{array}{ccccc}
           \{1,\ldots j\}&j+1& \cdots &n-1& n \\
            j-1 & j  & \ldots &n-2& n-1
          \end{array} \right)

          & \cdots  &

           \left( \begin{array}{ccccc}
           \{1,2\}&3& \cdots& \{p-1,\ldots n\} \\
            j-1 & j  & \ldots &n-2& n-1
          \end{array} \right)&

            \left( \begin{array}{ccccc}
           1&2& \cdots&p-1& \{p,\ldots n\} \\
             j-1 & j  & \ldots &n-2& n-1
          \end{array} \right) \\

           \left( \begin{array}{ccccc}
           \{1,\ldots j\}&j+1& \cdots &n-1& n \\
            j & j+1 & \ldots &n-1& n
          \end{array} \right) & \cdots  &

           \left( \begin{array}{ccccc}
           \{1,2\}&3& \cdots& \{p-1,\ldots n\} \\
            j & j+1 & \ldots &n-1& n

          \end{array} \right)&

           \left( \begin{array}{ccccc}
           1&2& \cdots&p-1& \{p,\ldots n\} \\
           j & j+1 & \ldots &n-1& n

          \end{array} \right)

     \end{array}$}\end{equation}

From the table above, we can see that for $p=1$, $|K_p|=n-p+1=n$, while for $2\leq p\leq n,\,$ $|K_p|=(n-p+1)^2$.

The next theorem gives us the cardinality of the semigroup $\textnormal{Reg}(\mathcal{OCT}_n)$.
\begin{theorem}\label{cadreg} Let $\mathcal{OCT}_n$ be as defined in equation \eqref{ctn}, then
$|\textnormal{Reg}(\mathcal{OCT}_n)|=\frac{n(n-1)(2n-1)+6n}{6}$.
\end{theorem}
\begin{proof} It is clear that $\textnormal{Reg}(\mathcal{OCT}_n)=K_1 \cup K_2 \cup \ldots \cup K_n$. Since this union is disjoint, we have that \begin{equation*}\begin{array}{c} |\textnormal{Reg}\mathcal{OCT}_n|=\sum_{p=1}^n|K_p|=|K_1|+\sum_{p=2}^n|K_p| = n+ \sum_{p=2}^n  (n-p+1)^2 \\ = n+(n-1)^2+(n-2)^2+ \cdots +2^2 +1^2 \\= \frac{n(n-1)(2n-1)+6n}{6},
\end{array}\end{equation*}\noindent as required.
\end{proof}

\begin{corollary}\label{cadreg2} Let $\mathcal{ORCT}_n$ be as defined in equation \eqref{orctn}. Then
$|\textnormal{Reg}(\mathcal{ORCT}_n)|=\frac{n(n-1)(2n-1)+6n}{3}-n$.
\end{corollary}

\begin{proof}
It follows from Theorem~\ref{cadreg} and the fact that $|\textnormal{Reg}(\mathcal{ORCT}_n)|=2|\textnormal{Reg}(\mathcal{OCT}_n)|-n$.
\end{proof}
 \subsection{The Semigroup $\textnormal{E}(\mathcal{ORCT}_n)$}
 Let $\alpha$ be in $\textnormal{E}(\mathcal{ORCT}_n)$, then it follows from  [\cite{am}, Lem. 13] that $\alpha$ is of the form \begin{equation}\label{alf} \alpha= \left( \begin{array}{cccccc}
            \{1,\ldots,i\} & i+1& i+2 & \ldots & i+j-1 & \{i+j, \ldots, n\} \\
            i & i+1 &  i+2& \ldots &i+j-1& i+j
          \end{array} \right).\end{equation} \noindent Since fix$(\alpha)=j+1$, then for each given domain set there will be only one corresponding image set.
          
Let
\begin{equation} E_p=\{\alpha \in \textnormal{E}(\mathcal{ORCT}_n) : |\textnormal{Im }\alpha|=p\} \quad (1\leq p\leq n). \end{equation}
To choose $\alpha\in E_p$ we only need to select the image set of $\alpha$ which is a $p$ consecutive(convex) numbers from the set $[n]$. Thus $|E_P|=n-p-1$. Consequently, we have the cardinality of the semigroup $\textnormal{E}(\mathcal{ORCT}_n)$.
\begin{theorem}\label{cidemp} Let $\mathcal{ORCT}_n$ be as defined in equation \eqref{orctn}. Then
$|\textnormal{E}(\mathcal{ORCT}_n)|=\frac{n(n+1)}{2}$.
\end{theorem}
\begin{proof} Following the  argument of the proof of Theorem \ref{cadreg} we have, \begin{equation*}\begin{array}{c} |\textnormal{E}(\mathcal{ORCT}_n)|=\sum_{p=1}^n|E_p|= \sum_{p=1}^n  (n-p+1) \\ = n+(n-1)+(n-2)+ \cdots +2 +1 \\= \frac{n(n+1)}{2}.
\end{array}\end{equation*}
\end{proof}
\begin{remark} Notice that idempotents in $\mathcal{ORCT}_n$ are necessarily order preserving, as such $|\textnormal{E}(\mathcal{OCT}_n)|=|\textnormal{E}(\mathcal{ORCT}_n)|= \frac{n(n+1)}{2}$.
\end{remark}

\section{Rank Properties}
In this section, we discuss some rank properties of the semigroups   $\textnormal{Reg}(\mathcal{ORCT}_n)$ and $\textnormal{E}(\mathcal{ORCT}_n)$.
 \subsection{Rank of $\textnormal{Reg}(\mathcal{OCT}_n)$}
 Just as in section 2 above, let us first consider the semigroup $\textnormal{Reg}(\mathcal{OCT}_n)$, the semigroup consisting of regular elements of order-preserving full contractions. Now, let $K_p$ be defined as in equation \eqref{j}. We have seen how elements of $K_p$ look like in Table \ref{tabl} above. Suppose we define: \begin{equation}\label{eta} \eta := \left( \begin{array}{ccccc}
           \{1,\ldots j\}&j+1& \cdots &n-1& n \\
            1 & 2  & \ldots &p-1& p
          \end{array} \right), \end{equation}
          \begin{equation}\label{delta} \delta :=  \left( \begin{array}{ccccc}
           1&2& \cdots&p-1& \{p,\ldots n\} \\
            1 &   2& \cdots&p-1 & p
          \end{array} \right) \end{equation} and
         \begin{equation}\label{tau} \tau:=  \left( \begin{array}{ccccc}
           1&2& \cdots&p-1& \{p,\ldots n\} \\
           j & j+1 & \ldots &n-1& n

          \end{array} \right) \end{equation}
that is, $\eta$ to be the  top left-corner element, $\delta$ be the top right-corner element while $\tau$ be the bottom right corner element in Table \ref{tabl}. And let $\textnormal{R}_\eta$ and $\textnormal{L}_\delta$ be the respective $\mathcal{R}$ and $\mathcal{L}$ equivalent classes of $\eta$ and $\delta$. Then for $\alpha$ in $K_p$ there exist two elements say $\eta'$ and $\delta'$ in $\textnormal{R}_\eta$ and $\textnormal{L}_\delta$, respectively for which $\alpha$  is $\mathcal{L}$ related to $\eta'$ and $\mathcal{R}$ related to $\delta'$ and that $\alpha=\eta'\delta'$. For the purpose of illustrations, consider
\begin{equation*} \alpha = \left( \begin{array}{ccccc}
           \{1,\ldots j-1\}&j&j+1& \cdots &\{n-1, n\} \\
            2 & 3&4  & \ldots &p+1
          \end{array} \right), \end{equation*} then the elements

    \begin{equation*} \left( \begin{array}{ccccc}
           \{1,\ldots j-1\}&j&j+1& \cdots &\{n-1, n\} \\
            1 & 2 &3 & \ldots & p
          \end{array} \right)\end{equation*} and
         \begin{equation*} \left( \begin{array}{ccccc}
           1&2& \cdots&p-1& \{p,\ldots n\} \\
            2 &   3& \cdots&p & p+1
          \end{array} \right)\end{equation*} are respectively elements of  $\textnormal{R}_\eta$ and $\textnormal{L}_\delta$ and that

\begin{equation*}\alpha =
     \left( \begin{array}{ccccc}
           \{1,\ldots j-1\}&j&j+1& \cdots &\{n-1, n\} \\
            1 & 2 &3 & \ldots & p
          \end{array} \right)
          \left( \begin{array}{ccccc}
           1&2& \cdots&p-1& \{p,\ldots n\} \\
            2 &   3& \cdots&p & p+1
          \end{array} \right). \end{equation*}
Consequently, we have the following lemma.
\begin{lemma}\label{jp} Let $\eta$ and $\delta$ be as defined  in equations \eqref{eta} and  \eqref{delta}, respectively. Then
$\langle \textnormal{R}_\eta \cup \textnormal{L}_\delta \rangle = K_p$.
\end{lemma}
 \begin{remark}\label{rtabl}The following are  observed from Table \ref{tabl}:
\begin{itemize}
\item[(i)] The element $\delta$ belongs to both $\textnormal{R}_\eta$ and $\textnormal{L}_\delta$;
\item[(ii)] $\tau\eta=\delta$;
\item[(iii)] For all $\alpha\in \textnormal{R}_\eta$, $\alpha\delta=\alpha$ while $\delta\alpha$ has rank less than $p$;
\item[(iv)] For all $\alpha\in \textnormal{L}_\delta$, $\delta\alpha=\alpha$ while $\alpha\delta$ has rank less than $p$;
\item[(v)]For all $\alpha,\beta\in \textnormal{R}_\eta\backslash \delta$ ( or $\textnormal{L}_\delta\backslash \delta$), rank($\alpha\beta)<p$.
\end{itemize}

\end{remark}

To investigate the rank of $\textnormal{Reg}(\mathcal{OCT}_n)$, let
\begin{equation}\label{lnp} L(n,p)=\{\alpha \in \textnormal{Reg}(\mathcal{OCT}_n) : |\textnormal{Im }\alpha|\leq p\} \quad (1\leq p\leq n), \end{equation}\noindent and let \begin{equation} Q_p=L(n,p)\backslash L(n,p-1). \end{equation} Then $Q_p$ is of the form $K_p \cup \{0\}$, where $K_p$ is the set of all elements of $\textnormal{Reg}(\mathcal{OCT}_n)$ whose height is exactly $p$. The product of any two elements in $Q_p$  say $\alpha$ and $\beta$ is of the form: \begin{equation*}\alpha\ast \beta = \left\{ \begin{array}{ll}
                                                    \alpha\beta, & \hbox{if $|h(\alpha\beta)|=p$;} \\
                                                      0, & \hbox{if $|h(\alpha\beta)|<p$}
                                                    \end{array} \right. \end{equation*}

  $Q_p$ is called the Rees quotient semigroup on $L(n,p)$. Next, we have the following lemma which follows from Lemma \ref{jp} and Remark \ref{rtabl}.
 \begin{lemma}\label{lrees}
  $(\textnormal{R}_\eta \cup \textnormal{L}_\delta)\backslash \delta$ is the minimal generating set for the Rees quotient semigroup  $Q_p$.
 \end{lemma}
To find the generating set for $L(n,p)$, we need the following proposition:
\begin{proposition}\label{prees} For $n\geq4,\,$ $ \langle K_p \rangle\,\subseteq \,\langle K_{p+1}\rangle$ for all $1\leq p\leq n-2$.
\end{proposition}
\begin{proof}
Let $\langle A \rangle=K_p$, to proof $\langle K_p \rangle\,\subseteq \,\langle K_{p+1}\rangle$, it suffices to show that $A\subseteq \langle K_{p+1}\rangle$.  From Lemma \ref{lrees}  $A= (\textnormal{R}_{\eta} \cup \textnormal{L}_{\delta} )\backslash {\delta}$. Now, let $\alpha$ be in $A$:

CASE I: If $\alpha=\eta$,
          then $\alpha$ can be written as $\alpha=$

          \begin{equation*}\resizebox{1\textwidth}{!}{$ \left( \begin{array}{cccccc}
           \{1,\ldots j-1\}&j&j+1& \cdots &n-1& n \\
            j-2 & j-1&j  & \cdots&n-2 &n-1
          \end{array} \right) \left( \begin{array}{cccccc}
           \{1,\ldots j-1\}&j&j+1& \cdots &n-1& n\\
            1 & 2&3  & \cdots&p &p+1
          \end{array} \right),$}
\end{equation*} a product of two elements of $K_{p+1}$.

CASE II: If $\alpha\in \textnormal{R}_{\eta}\backslash \eta$, then $\alpha$ is of the form  \begin{equation*}\left( \begin{array}{ccccc}
           \{1,\ldots j-k\}&j-k+1& \cdots &n-2 &\{n-k,\ldots, n\} \\
            1 & 2 & \cdots&p-1 &p
          \end{array} \right), \, (k=1,2,\dots,j-2).\end{equation*} Then $\alpha $ can be written as:
          \begin{equation*}\resizebox{1\textwidth}{!}{$ \left( \begin{array}{cccc}
           \{1,\ldots, j-k-1\}&j-k & \cdots &\{n-k,\ldots, n\} \\
            j-k-1 & j-k  & \cdots &n-k
          \end{array} \right) \left( \begin{array}{ccccc}
           \{1,\ldots j-k\}&j-k+1& \cdots &n-k& \{n-k+1,\ldots,n\}\\
            1 & 2  & \cdots&p &p+1
          \end{array} \right),$}
\end{equation*} a product of two elements of $K_{p+1}$.

CASE III: If $\alpha\in \textnormal{L}_{\delta}\backslash \delta$, then $\alpha$ is of the form  \begin{equation*}\left( \begin{array}{ccccc}
           1&2& \cdots&p-1& \{p,p+1,\ldots n\} \\
            r &   r+1& \cdots& p+r-2 & p+r-1
          \end{array} \right),\, (r=2,3,\ldots, n-p+1)\end{equation*} and it can be written as:

\begin{equation*}\resizebox{1\textwidth}{!}{$
\left( \begin{array}{ccccc}
           1&2& \cdots&p& \{p+1,\ldots n\} \\
            2 &   3& \cdots&p+1 &p+2
          \end{array} \right)
         \left( \begin{array}{ccccc}
           1&2& \cdots&p& \{p+1,\ldots n\} \\
            r-1 &   r& \cdots&p+r-2 & p+r-1
          \end{array} \right),$}
\end{equation*}   hence the proof.
\end{proof}
\begin{remark}\label{rrank} Notice that by the proposition above, the generating set for $Q_p$ ($1\leq p\leq n-1$) generates the whole $L(n, p)$.
\end{remark}

The next theorem gives us the rank of the subsemigroup $L(n,p)$ for $1\leq p\leq n-1$.
\begin{theorem}\label{trank} Let $L(n,p)$ be as defined in equation \eqref{lnp}. Then
for $n\geq 4$ and $1<p\leq n-1$,  the rank of $L(n,p)$ is $2(n-p)$.
\end{theorem}
\begin{proof}
It follows from Lemma \ref{lrees} and Remark \ref{rrank} above.
\end{proof}
Now as a consequence, we readily have the following corollaries.
 \begin{corollary}\label{cr1} Let $L(n,p)$ be as defined in equation \eqref{lnp}. Then
 the rank of $L(n,n-1)$ is 2.
 \end{corollary}
 \begin{corollary}\label{cr2} Let $\mathcal{OCT}_n$ be as defined in equation \eqref{ctn}. Then
 the rank of $\textnormal{Reg}(\mathcal{OCT}_n)$ is 3.
 \end{corollary}
 \begin{proof} The proof follows from Corollary \ref{cr1} coupled with the fact that $\textnormal{Reg}(\mathcal{OCT}_n)= L(n,n-1)\cup id_{[n]}$, where $id_{[n]}$ is the identity element on $[n]$.
 \end{proof}

\subsection{Rank of $\textnormal{Reg}(\mathcal{ORCT}_n)$}
To discuss the rank of  $\textnormal{Reg}(\mathcal{ORCT}_n)$, consider the Table \ref{tabl} above. Suppose we reverse the order of the image set of elements in that table, then we will have the set of order-reversing elements of $\textnormal{Reg}(\mathcal{ORCT}_n)$. For  $1\leq p\leq n$, let
\begin{equation}J_p=\{\alpha \in \textnormal{Reg}(\mathcal{ORCT}_n) : |\textnormal{Im }\alpha|= p\} \end{equation} and let
\begin{equation}K_p^*=\{\alpha \in J_p : \alpha \textrm{ is order-reversing} \}. \end{equation}

Observe that $J_p= K_p \cup K_p^*$. Now define:  \begin{equation}\label{eta2} \eta^* = \left( \begin{array}{ccccc}
           \{1,\ldots j\}&j+1& \cdots &n-1& n \\
            p & p-1  & \ldots & 2 & 1
          \end{array} \right), \end{equation}
          \begin{equation}\label{delta2} \delta^* =  \left( \begin{array}{ccccc}
           1&2& \cdots&p-1& \{p,\ldots n\} \\
            p &  p-1 & \cdots& 2 & 1
          \end{array} \right) \end{equation} and
         \begin{equation}\label{tau2} \tau^* =  \left( \begin{array}{ccccc}
           1&2& \cdots&p-1& \{p,\ldots n\} \\
           n & n-1 & \ldots & j+1 & j

          \end{array} \right) \end{equation}

          i.e., $\eta^*, \delta^*$ and $\tau^*$ are respectively $\eta, \delta$ and $\tau$ with image order-reversed.

       \begin{remark}
       Throughout this section, we will write $\alpha^*$ to mean a mapping in $K_p^*$ which has a corresponding mapping $\alpha$ in $K_p$ with order-preserving image.
       \end{remark}
          And let $R_{\eta^*}$ and $L_{\delta^*}$ be the respective $\mathcal{R}$ and $\mathcal{L}$ equivalent classes of $\eta$ and $\delta$. Then we have the following lemmas which are analogue to Lemma \ref{jp}.

\begin{lemma}\label{jp2} Let $\eta$ and $\delta^*$ be as defined  in equations \eqref{eta} and \eqref{delta2}, respectively. Then
$\langle \textnormal{R}_\eta \cup \textnormal{L}_{\delta^*} \rangle = K_p^*$.
\end{lemma}
 \begin{proof}
 Let $ \alpha^*= \left( \begin{array}{ccccc}
            \{1,\ldots,a+1\} & a+2 & \ldots & a+{p-1} & \{a+p,\ldots,n\} \\
            x+p & x+{p-1} & \ldots &x+2& x+1
          \end{array} \right)$ be in $K_p^*$, then there exists  $\alpha\in K_p$ such that by Lemma \ref{jp}, $\alpha$ can be expressed as the following product:

          \begin{equation*}  \left( \begin{array}{ccccc}
            \{1,\ldots,a+1\} & a+2& \ldots & a+{p-1} & \{a+p,\ldots,n\} \\
            y+1 & y+2 & \ldots &y+{p-1}& y+p
          \end{array} \right) \left( \begin{array}{ccccc}
            \{1,\ldots,b+1\} & b+2 & \ldots & b+{p-1} & \{b+p,\ldots,n\} \\
            x+1 & x+2 &  \ldots &x+{p-1}& x+p
          \end{array} \right)\end{equation*} a product of elements of  $\textnormal{R}_\eta$ and $\textnormal{L}_\delta$, respectively. Therefore,
           $\alpha^*$ can be expressed as the following product:
          \begin{equation*}  \left( \begin{array}{ccccc}
            \{1,\ldots,a+1\} & a+2& \ldots & a+{p-1} & \{a+p,\ldots,n\} \\
            y+1 & y+2 & \ldots &y+{p-1}& y+p
          \end{array} \right) \left( \begin{array}{ccccc}
            \{1,\ldots,b+1\} & b+2 & \ldots & b+{p-1} & \{b+p,\ldots,n\} \\
            x+1 & x+2 &  \ldots &x+{p-1}& x+p
          \end{array} \right)\end{equation*} a product of elements of  $\textnormal{R}_\eta$ and $\textnormal{L}_{\delta^*}$, respectively.
 \end{proof}

 \begin{lemma}\label{jp3}  Let $J_p=\{\alpha \in \textnormal{Reg}(\mathcal{ORCT}_n) : |\textnormal{Im }\alpha|= p\}$. Then, $\langle R_\eta \cup L_{\delta^*} \rangle = J_p$.
\end{lemma}
\begin{proof}
Since $J_p= K_p \cup K_p^*$, to proof $\langle R_\eta \cup L_{\delta^*} \rangle = J_p$,  is suffices  by Lemma \ref{jp3} to show that $K_p \subseteq\langle K_p^* \rangle$. Now, let  $$\alpha= \left( \begin{array}{ccccc}
            \{1,\ldots,a+1\} & a+2& \ldots & a+{p-1} & \{a+p,\ldots,n\} \\
            b+1 & b+2 & \ldots &b+{p-1}& b+p
          \end{array} \right)$$ \noindent be in $K_p$, if $\alpha$ is an idempotent, then there exists $\alpha^* \in K_p^*$ such that $(\alpha^*)^2=\alpha.$ Suppose $\alpha$ is not an idempotent, define $$\epsilon= \left( \begin{array}{cccccc}
            \{1,\ldots,b+1\} & b+2& b+3 & \ldots & b+{p-1} & \{b+p,\ldots,n\} \\
            b+1 & b+2 &  b+3& \ldots &b+{p-1}& b+p
          \end{array} \right)$$ \noindent which is an idempotent in $K_p$, then $\alpha$ can be written as $\alpha=\alpha^*\epsilon^*$.
\end{proof}
Before stating the main theorem of this section, let
\begin{equation}\label{mp} M(n,p)=\{\alpha \in \textnormal{Reg}(\mathcal{ORCT}_n) : |\textnormal{Im }\alpha|\leq p\} \quad (1\leq p\leq n). \end{equation}
  And let \begin{equation} W_p=M(n,p)\backslash M(n,p-1) \end{equation} be Rees quotient semigroup on $M(n,p)$. From Lemma \ref{jp3} and Remark \ref{rtabl} we have:
\begin{lemma}\label{lrees2}
  $(\textnormal{R}_\eta \cup \textnormal{L}_{\delta^*})\backslash \delta$ is the minimal generating set for the Rees quotient semigroup  $W_p$.
 \end{lemma}
The next proposition is also analogue to Proposition \ref{prees} which plays an important role in finding the generating set for the subsemigroup $M(n,p)$.
\begin{proposition}\label{prees2} For $n\geq4,\; \langle J_p \rangle\,\subseteq \,\langle J_{p+1}\rangle$ for all $1\leq p\leq n-2$.
\end{proposition}
\begin{proof}
The proof follows the same pattern as the proof of the Proposition \ref{prees}.
We want to show that $(\textnormal{R}_\eta \cup \textnormal{L}_{\delta^*}  )\subseteq \,\langle J_{p+1}\rangle$ and by Proposition \ref{prees} we only need to show that $\textnormal{L}_{\delta^*} \subseteq \,\langle J_{p+1}\rangle$. Now  Let  $\alpha$ be in  $\textnormal{L}_{\delta^*}$,

Case I: $\alpha\in \textnormal{L}_{\delta^*}\backslash \tau^* $, then $\alpha$ is the of the form \begin{equation*}\left( \begin{array}{ccccc}
           1&2& \cdots&p-1& \{p,p+1,\ldots n\} \\
            p+r-1 &   p+r-2& \cdots& r+1& r
          \end{array} \right)\; (r=1,2,\ldots, n-p),\end{equation*} and it can be written as

\begin{equation*}\resizebox{1\textwidth}{!}{$\alpha=
\left( \begin{array}{ccccc}
           1&2& \cdots&p& \{p+1,\ldots n\} \\
            2 &   3& \cdots&p+1 &p+2
          \end{array} \right)
         \left( \begin{array}{ccccc}
           1&2& \cdots&p& \{p+1,\ldots n\} \\
            p+r &   p+r-1& \cdots& r+1& r
          \end{array} \right),$}
          \end{equation*} a product of two elements of $J_{p+1}$.

Case II:  $\alpha=\tau^*$ then $\alpha$ can be written as
\begin{equation*}\alpha=
\left( \begin{array}{ccccc}
           1&2& \cdots&p-1& \{p,\ldots n\} \\
            1 &   2& \cdots&p-1 &p
          \end{array} \right)
         \left( \begin{array}{ccccc}
           1&2& \cdots&p& \{p+1,\ldots n\} \\
            n &   n-1& \cdots& j& j-1
          \end{array} \right).
          \end{equation*}
The first element in the product above is $\delta \in J_p$, but it was shown in Remark~\ref{rtabl} that it can be written as $\tau\eta$ which were both shown in Proposition \ref{prees} that they can be expressed as product of elements of $J_{p+1}$. Hence the proof.
\end{proof}

\begin{remark}
Notice also that, by Proposition \ref{prees2} above, for $2\leq p\leq n-1$ the generating set for $W_p$ generates the whole $M(n, p)$
\end{remark}
The next theorem gives us the rank of subsemigroup $M(n,p)$ for $2\leq p\leq n-1$.
\begin{theorem}\label{trank2} Let $M(n, p)$ be as defined in equation \eqref{mp}. Then
for $n\geq 4$ and $2<p\leq n-1$,  the rank of $M(n,p)$ is $2(n-p)+1.$
\end{theorem}
\begin{proof}
To proof this, we only need to compute the cardinality of the set  $(R_\eta \cup L_{\delta^*})\backslash \delta$, which from Table~\ref{tabl} we easily obtain $(n-p)+(n-p)+1=2(n-p)+1$.
\end{proof}
As a consequence, we have the following corollaries.
 \begin{corollary}\label{cr3}Let $M(n, p)$ be as defined in equation \eqref{mp}. Then
 the rank of $M(n,n-1)$ is 3.
 \end{corollary}
 \begin{corollary}\label{cr4} Let $\mathcal{ORCT}_n$ be as defined in equation \eqref{orctn}. Then
 the rank of $\textnormal{Reg}(\mathcal{ORCT}_n)$ is 4.
 \end{corollary}
 \begin{proof} The only elements of $\textnormal{Reg}(\mathcal{OCT}_n)$ that are not in $ M(n,n-1)$ are the elements $\{id_{[n]},id_{[n]}^*\}$. It is clear that $(id^*)^2=id$.
 \end{proof}

\bibliographystyle{amsplain}

\bigskip
\bigskip

\bigskip
\bigskip
\bigskip
\bigskip
\bigskip
\bigskip
\bigskip
\bigskip

{\footnotesize {\bf First Author}\; \\ {Department of
Mathematics}, {Bayero University Kano, P.O.Box 3011,} {Kano, Nigera.}\\
{\tt Email: mmzubairu.mth\@ buk.edu.ng}\\

{\footnotesize {\bf Second Author}\; \\ {Khalifa University, P. O. Box 127788, Sas al Nakhl, Abu Dhabi, UAE}\\
{\tt Email: abdullahi.umar@ku.ac.ae}\\

{\footnotesize {\bf Third Author}\; \\ {Department of Mathematics, and Computer Sciences, Sule Lamido University, Kafin Hausa
,  Nigera.}\\
{\tt Email: muhammadaliyu2@nda.edu.ng}\\

\end{document}